\documentclass{amsart}
\usepackage{amsthm, amsmath,textcomp,gensymb,amssymb,color,graphics,amstext,amscd, txfonts, epsfig, psfrag, color, multicol, graphicx}
\usepackage[normalem]{ulem}
\usepackage[colorlinks=true, pdfstartview=FitV, linkcolor=blue, citecolor=blue, urlcolor=blue]{hyperref}

\newtheorem{thm}{Theorem}

\newtheorem{lemma}[thm]{Lemma}

\newtheorem{defn}{Definition}

\theoremstyle{definition}

\definecolor{magenta}{rgb}{.5,0,.5} 
\definecolor{dred}{rgb}{.5,0,0} 
\definecolor{green}{rgb}{0,.5,0} 
\definecolor{blue}{rgb}{0,0,0.5} 
\definecolor{black}{rgb}{0,0,0} 
\definecolor{dgreen}{rgb}{0,.3,0} 
\definecolor{vdred}{rgb}{.3,0,0} 
\definecolor{red}{rgb}{1,0,0} 
\definecolor{gray}{rgb}{.5,.5,.5} 
\definecolor{cerulean}{rgb}{0,.48,.65}

\def\ni{\noindent}

\def\ms{\medskip}

\def\onto{{\kern3pt\to\kern-8pt\to\kern3pt}}
\def\<{\langle}
\def\>{\rangle}
\def\|{{\ |\ }}

\def\*{^{\star}}

\newcommand{\executeiffilenewer}[3]{%
\ifnum\pdfstrcmp{\pdffilemoddate{#1}}%
{\pdffilemoddate{#2}}>0%
{\immediate\write18{#3}}\fi%
}

\address{Department of Mathematics \\ Cornell University\\
Ithaca, NY 14853--4201 \\ USA}

\email{{\tt yl763@cornell.edu}}
\begin{document}

\title{A Hyperbolic group with a finitely presented subgroup that is not of type 
$\textup{FP}_3$.}

\author{Yash Lodha}

\date \today

\begin{abstract}
\ni  Brady proved that there are hyperbolic groups with finitely 
presented subgroups that are not of type $\textup{FP}_3$ (and hence not hyperbolic).
We reprove Brady's theorem by presenting a new construction.
Our construction uses Bestvina-Brady Morse theory, 
but does not involve branched coverings.
 \ms

\footnotesize{\ni \textbf{2010 Mathematics Subject
Classification:  20F67}  \\ \ni \emph{Key words and phrases:} 
hyperbolic group, finiteness properties}
\end{abstract}
\maketitle
\section{Introduction}

Hyperbolic groups were introduced by Gromov \cite{Gromov} 
as a generalization of fundamental groups of negatively curved manifolds
and of finitely generated free groups. 
A finitely generated group is hyperbolic if its Cayley graph 
is hyperbolic as a metric space. 
Hyperbolic groups are finitely presented, 
and have solvable word and conjugacy problems.

The property of being hyperbolic is not inherited by subgroups. 
For instance finitely generated free groups have infinitely generated subgroups which 
cannot be hyperbolic since hyperbolic groups are finitely presentable. 
It is natural to then ask if the property of being hyperbolic is inherited 
by finitely generated subgroups. 
Rips constructed the first examples of finitely generated 
subgroups of hyperbolic groups that are not finitely presentable.\cite{Rips}

\begin{thm}{(Rips)}
There are hyperbolic groups with subgroups that are finitely generated but not finitely 
presentable.
\end{thm}

So then one asks whether finitely presented subgroups inherit hyperbolicity.
Gromov gave what appeared to be an example of a non-hyperbolic 
finitely presented subgroup of a hyperbolic group. 
But it was later discovered by Bestvina that the ambient group of this
example is not hyperbolic \cite{Brady}. 
In $1999$ Noel Brady constructed the first example of a hyperbolic group 
with a finitely presented subgroup that is not hyperbolic. 
In particular, he showed the following. (Theorem $6.1$ in \cite{Brady}.) 

\begin{thm}{(Brady)}
There exists a piecewise Euclidean cubical complex $Y$ and a continuous map 
$f:Y\to S^1$ with the following properties.
\begin{enumerate}
 \item The image of $\phi=f_{*}:\pi_1(Y)\to \pi_1(S^1)$ is of finite index in $\pi_1(S^{1})$.
 \item The universal cover $X=\widetilde{Y}$ is a $\textup{CAT}(0)$ metric space.
 \item $X$ has no embedded flat planes.
 \item The map $f$ lifts to a $\phi$-equivariant Morse function $\tilde{f}:X\to \mathbb{R}$
 whose ascending and descending links are all homeomorphic to $S^2$.
\end{enumerate}
\end{thm}

The cube complex $Y$ is constructed as a branched covering of a product of graphs.
A consequence of this is the following theorem. (Theorem $1.1$ in \cite{Brady})

\begin{thm}{(Brady)}
 There exists a short exact sequence of groups
 $$1\to H\to G\to \mathbb{Z}\to 1$$ such that
 \begin{enumerate}
  \item $G$ is torsion free hyperbolic.
  \item $H$ is finitely presented.
  \item $H$ is \emph{not} of type $FP_3$.
 \end{enumerate}
 In particular, $H$ is not hyperbolic.
\end{thm}

The group $G=\pi_1(Y)$ is the fundamental group of a nonpositively curved cube complex whose 
universal cover contains no flat planes. 
So by a theorem of Bridson, Eberlin and Gromov 
\cite{Bridson} $\widetilde{Y}$ admits a hyperbolic metric.
Since $G$ acts properly, cocompactly and by isometries on $\widetilde{Y}$, it is a hyperbolic group. 
The subgroup $H$ emerges as the kernel of the induced homomorphism to $\mathbb{Z}$. 
Brady uses Bestvina--Brady Morse theory \cite{BeBr} to prove 
the finiteness properties of $H$.
Brady asks if there are examples of hyperbolic groups with subgroups of 
type $F_n$ but not type $F_{n+1}$ for all $n\in \mathbb{N}$ \cite{Brady}. 

We shall reprove the results of Rips and Brady using slightly different methods.
In Section $3$ we present an example of a finitely generated but not finitely presentable 
subgroup of a $\textup{CAT}(-1)$ group.
Reading this construction will prepare the reader for the construction 
in Section $4$.
In Section $4$, we shall construct a hyperbolic group with a 
finitely presented subgroup which is not of type $FP_3$. 
Our cube complexes will emerge as subcomplexes of products of graphs,
and do not require branched coverings. 

\section{Preliminaries}

\subsection{Nonpositively curved metric spaces and groups.}

Let $G$ be a finitely generated group with a finite generating set $A$
that is closed under inverses. 
Recall that the \emph{cayley graph} $\Gamma_{G,A}$ of the group 
with respect to $A$ is formed as follows:
The vertices of the graph are elements of the group and for every
$g\in G$ and $a\in A$, the vertices $g,ga$ are connected by an edge.
This forms a metric space by declaring that each edge is isometric to
the unit interval and the metric on the graph is the induced path metric.
This is a \emph{geodesic} metric space, i.e. for any pair of points 
$x,y\in \Gamma_{G,A}$ there is an isometric embedding $\phi:[0,r]\to \Gamma_{G,A}$
such that $\phi(0)=x,\phi(r)=y$ and $r=d(x,y)$.

Given a geodesic metric space $X$,
a triple $x,y,z\in X$ forms a \emph{geodesic} triangle $\triangle(x,y,z)$
obtained by joining $x,y,z$ pairwise by geodesics.
Let $\triangle_{\mathbb{E}^2}(x^{\prime},y^{\prime},z^{\prime})$ be a Euclidean triangle
such that the Euclidean distance between each pair $(x^{\prime},y^{\prime}),
(x^{\prime},z^{\prime}),(y^{\prime},z^{\prime})$
is equal to the distance between the corresponding pair $(x,y),(x,z),(y,z)$.
The triangle $\triangle_{\mathbb{E}^2}(x^{\prime},y^{\prime},z^{\prime})$
is called a \emph{comparison triangle} for $\triangle(x,y,z)$.
If $p$ is a point on the geodesic joining $x,y$, there 
is a point $p^{\prime}$ in the geodesic joining $x^{\prime},y^{\prime}$
such that $d_X(x,p)=d_{\mathbb{E}^2}(x^{\prime},p^{\prime})$.
This is called a \emph{comparison point}. 

$X$ is said to be $\textup{CAT}(0)$ if for any geodesic triangle 
$\triangle(x,y,z)$ and a pair $p,q\in \triangle(x,y,z)$, 
the corresponding comparison points 
$p^{\prime},q^{\prime}\in \triangle_{\mathbb{E}^n}(x^{\prime},y^{\prime},z^{\prime})$
have the property that $d_X(p,q)\leq d_{\mathbb{E}^2}(p^{\prime},q^{\prime})$.

Similarly we define $\textup{CAT}(-1)$ spaces by replacing Euclidean space
in the above definition by a complete, simply connected,
Riemannian $2$-manifold of constant curvature $-1$.
A group is said to be $\textup{CAT}(0)$ (or $\textup{CAT}(-1)$) if it acts
properly, cocompactly and by isometries on a $\textup{CAT}(0)$ 
(or respectively a $\textup{CAT}(-1)$) space.

A path connected metric space $X$ is said to be hyperbolic if there is a $\delta>0$
such that for any geodesic triangle $\triangle (x,y,z)$ and a point $p$ 
on the geodesic connecting $x,y$ there is a point $q$ in the union of the geodesics
connecting $y,z$ and $x,z$ such that $d_X(p,q)<\delta$.
This property is a quasi isometry invariant.
A group is said to be hyperbolic if its cayley graph
(with respect to any generating set) is hyperbolic as a metric space.

A comprehensive introduction and survey of $\textup{CAT}(0)$, $\textup{CAT}(-1)$ and hyperbolic
spaces can be found in \cite{BridsonH}.

\subsection{Cube Complexes and nonpositive curvature}

By a $\emph{regular }n\emph{-cube}$ $\square^n$ 
we mean a cube in $\mathbb{R}^n$ which is isometric 
to the cube $[0,1]^n$ in $\mathbb{R}^n$. 
Informally a cube complex is a cell complex of 
regular Euclidean cubes glued along their faces by isometries. 
More formally, a cube complex is a cell complex $X$ that satisfies the following conditions.

\begin{enumerate}
 \item For each $n$-cell $e$ in $X$ there is an isometry $\chi_e:\square^n\to e$.
 
 \item  A map $f:\square^n\to e$ is an \emph{admissible characteristic function} 
 if it is $\chi_e$ precomposed with a partial isometry of $\mathbb{R}^n$. 
 For any cell $e$ in $X$ the restriction of any $\chi_e$ 
 to a face of $\square^n$ is an admissible characteristic function of a cell of $C$.
\end{enumerate}

The metric on such cube complexes is the piecewise Euclidean metric (see \cite{BridsonH}). 

\begin{defn}
 Given a face $f$ of a regular cube $\square^n$, let $x$ be the center of this face. 
 The link $Lk(f,\square^n)$ is the set of unit tangent vectors at 
 $x$ that are orthogonal to $f$ and point in $\square^n$. 
 This is a subset of the unit sphere $S^{n-1}$ which is 
 homeomorphic to a simplex of dimension $n-\textup{dim}(f)-1$. 
 This admits a natural spherical metric, in which the dihedral angles are right angles. 
 \end{defn}
 
\begin{defn}
 Let $f$ be a cell in $X$. 
 Let $S=\{C\mid C\textup{ is a cube in }X\textup{ that contains }f\textup{ as a face }\}$. 
 The link $Lk(f,X)=\bigcup_{C\in S}Lk(f,C)$. 
 This is a complex of spherical ``all right'' simplices glued along their faces by isometries.
 This admits a natural piecewise spherical metric.
\end{defn}
Gromov gave the following characterizations of $\textup{CAT}(0)$ and 
$\textup{CAT}(-1)$ cube complexes by combinatorial conditions 
on the links of vertices in \cite{Gromov}. 
A nice survey of these results can be found in \cite{Davis} and \cite{Davis2} 
(see Proposition $I.6.8$).

\begin{defn}
 A simplicial complex is said to satisfy the \emph{no $\square$-condition} 
 if there are no $4$-cycles in the $1$-skeleton for which none 
 of the pairs of opposite vertices of the cycle
 are connected by an edge. 
 A simplicial complex $Z$ is called a ``flag'' complex 
 if any set $v_1,...,v_n$ of vertices of $Z$ that are pairwise 
 connected by an edge span a simplex. 
 This is also known as the ``no empty triangles'' condition.
\end{defn}

\begin{defn}
 A cube complex $X$ is said to be nonpositively curved 
 if the link of each vertex is a flag complex. 
\end{defn}

\begin{thm}\label{npc}
 (Gromov) A cube complex $X$ is $\textup{CAT}(0)$ 
 if and only if it is nonpositively curved and simply connected. 
 Furthermore, $X$ admits a $CAT(-1)$ metric if and only 
 if it is $\textup{CAT}(0)$ and the link of each vertex satisfies the no $\square$-condition.
\end{thm}

The following is a characterization of $\textup{CAT(0)}$ metric spaces that are also
hyperbolic.\cite{Bridson}

\begin{thm}\label{flatplanes}
 (Gromov, Eberlin, Bridson) A $\textup{CAT}(0)$ metric space with a cocompact group of 
 isometries is 
 hyperbolic if and only if it does not contain isometrically embedded flat planes.
\end{thm}

\subsection{Topological finiteness properties of groups.}

The classical finiteness properties of groups are that of 
being finitely generated and finitely presented. 
These notions were generalized by C.T.C. Wall \cite{Wall}. 
In this paper we are concerned with the properties \emph{type $F_n$} 
and \emph{type $FP_{n}$}. 
These properties are quasi-isometry invariants of groups \cite{Alonso}. 
In order to discuss the property type $F_n$ 
first we need to define Eilenberg-Maclane complexes.

\begin{defn}
  An Eilenberg-Maclane complex for a group $G$, 
  or a $K(G,1)$, is a connected CW-complex $X$ such that 
  $\pi_1(X)=G$ and $\widetilde{X}$ is contractible.
 \end{defn}
 
 It is a fact that for any group $G$, 
 there is an Eilenberg-Maclane complex $X$ which is unique up to homotopy type.
 A group is said to be \emph{of type $F_n$} if it has an 
 Eilenberg Maclane complex with a finite $n$-skeleton.  
 Clearly, a group is finitely generated if and only if it is of type $F_1$, 
 and finitely presented if and only if it is of type $F_2$.
 (For more details see \cite{Geoghegan}.)
 
Torsion-free hyperbolic groups are of 
type $F_{\infty}$, which means that they are of type $F_n$ 
for all $n\in \mathbb{N}$. 
This follows from the following result of Rips that appears in \cite{Rips}.

\begin{thm}
 (Rips) Let $H$ be a hyperbolic group. Then there exists a locally finite, simply connected, 
 finite dimensional simplicial complex on which $H$ acts faithfully, 
 properly, simplicially and cocompactly. 
 In particular, if $H$ is torsion free, then the action is free and
 the quotient of this complex by $H$ is a finite Eilenberg-Maclane complex $K(H,1)$.
\end{thm}

\subsection{Homological finiteness properties of groups.}

For a group $G$, consider the group ring $\mathbb{Z}G$.
We view $\mathbb{Z}$ as a $\mathbb{Z}G$ module where the action of $G$ is trivial,
i.e. $g\cdot 1=1$ for every $g\in G$.
A module is called \emph{projective} if it is the direct summand of a free module.
The group $G$ is said to be of type $FP_n$ 
if there is a projective $\mathbb{Z}G$-resolution (an exact sequence):
$$...\to P_n\to P_{n-1}\to...\to P_0\to \mathbb{Z}$$
of the trivial $\mathbb{Z}G$ module $\mathbb{Z}$ such that for each $1\leq i\leq n$ 
$P_i$ is finitely generated as a $\mathbb{Z}G$ module.

A group is of type $FP_1$ if and only if it is finitely generated.
However, there are examples of groups that are of type $FP_2$ but not
finitely presented.
If a group is of type $F_n$ then it is of type $FP_n$.
In general for $n>1$ type $FP_n$ does not imply type $F_n$.
(There are examples due to Bestvina-Brady \cite{BeBr}.)
However, whenever a group is finitely presented and of type $FP_n$, it is also 
of type $F_n$.
For a detailed exposition about homological finiteness properties, we refer
the reader to \cite{Geoghegan}.

\subsection{Bestvina-Brady Morse theory}

Here we shall sketch the main tool used in this paper. 
Bestvina--Brady Morse theory was introduced in \cite{BeBr} 
to study finiteness properties of subgroups of certain 
right angled Artin groups. Morse theory is defined more generally for affine cell complexes,
but we shall only discuss the special case of piecewise euclidean cube complexes.

Let $X$ be a piecewise Euclidean cube complex. 
Let $G$ act freely, properly, cocompactly, cellularly and by isometries on $X$.
Let $\mathbb{Z}$ act on $\mathbb{R}$ in the usual way. 
Let $\phi:G\to \mathbb{Z}$ be a homomorphism, and 
let $H=Ker(\phi)$.
We fix these assumptions for the rest of this subsection.

\begin{defn}
 A $\phi$-equivariant Morse function is a map $f:X\to \mathbb{R}$ that satisfies,
 
 \begin{enumerate}
  \item For any cell $F$ of $X$ (with the characteristic function $\chi_F:\square^n\to F$),
   the composition $f\circ \chi_F$ is the restriction of a
   nonconstant affine map $\mathbb{R}^n\to \mathbb{R}$.
  
  \item The image of the $0$-skeleton is a discrete subset of $\mathbb{R}$.
  
  \item $f$ is $G$-equivariant, i.e for all $g\in G, x\in X$, $f(g\cdot x)=\phi(g)\cdot f(x)$. 
 \end{enumerate}
 
 \end{defn}
 
 One can think of a Morse function as a height function on $X$. The kernel 
 $H=Ker(\phi)$ acts on the level sets of $X$, i.e. 
 inverse images $f^{-1}(x)$ for $x\in \mathbb{R}$.
 Topological properties of level sets can be used to deduce the finiteness properties of $H$. 
 In \cite{BeBr} it was shown
 that the topological properties of the level sets are determined 
 by the topology of links of vertices. We make this precise below.
 
\begin{defn}
 Given a vertex $v$ of $X$, 
 the ascending link $Lk^{\uparrow}(v,X)$ is defined as 
 $$Lk^{\uparrow}(v,X)=\bigcup \{Lk(v,F)\mid v\textup{ is the minimum of }f\circ \chi_F\}$$
 
 Similarly, the descending link $Lk^{\downarrow}(v,X)$ is defined as 
 $$Lk^{\downarrow}(v,X)=\bigcup \{Lk(v,F)\mid v\textup{ is the maximum of }f\circ \chi_F\}$$
\end{defn}

We now summarize the main result of Bestvina--Brady Morse theory. 
For the details of the proof see \cite{BeBr} or \cite{Brady}.

\begin{thm}\label{mt}
(Bestvina, Brady)
Consider the situation described above.
Then the following holds:

\begin{enumerate}
 \item If for every vertex $v\in X$, $Lk^{\uparrow}(v,X),Lk^{\downarrow}(v,X)$ 
 are simply connected, then $H$ is finitely presented. 
 
 \item If for every vertex $v\in X$, 
 $\widetilde{H}_k(Lk^{\uparrow}(v,X),\mathbb{Z}),
 \widetilde{H}_k(Lk^{\downarrow}(v,X),\mathbb{Z})=0$ 
 for every $k$ such that either $k=n+1$ or $1\leq k\leq n-1$, and 
$\widetilde{H}_n(Lk^{\uparrow}(v,X),\mathbb{Z}),\widetilde{H}_n(Lk^{\downarrow}(v,X),\mathbb{Z})
\neq 0$ then $H$ is of type $FP_n$ but not of type $FP_{n+1}$.
\end{enumerate}
\end{thm}

We remark that if a group is finitely presented and of type $FP_n$,
then it is of type $F_n$.
(See \cite{Geoghegan}.)

\section{A $\textup{CAT}(-1)$ example}

In this section we produce a 
square complex $X$ with a map 
$f:X\to S^1$ which lifts to an $f_*$-equivariant 
Morse function $\widetilde{f}:\widetilde{X}\to \mathbb{R}$
with the properties: 

\begin{enumerate}
 \item The link of each vertex is a finite graph with no $3$ or $4$ cycles,

 \item The ascending and descending links of all vertices are homeomorphic to $S^1$.
 \end{enumerate}
 
 First we define a bipartite graph $\Gamma$ which will be an
 ingredient in both constructions.
 
 \begin{defn}\label{Gamma}
 $\Gamma$ is the following graph:
 
  The vertex set: $V(\Gamma)=A_+\cup A_-\cup B_-\cup B_+$ where,
  \begin{enumerate}
   \item $A_+=\{a_0^+,a_1^+,...,a_{10}^+\}$.
   \item $A_-=\{a_0^-,a_1^-...,a_{10}^-\}$.
   \item $B_+=\{b_0^+,b_1^+,...,b_{10}^+\}$.
   \item $B_-=\{b_0^-,b_1^-,...,b_{10}^-\}$.
  \end{enumerate}
 The edge set: $E(\Gamma)=E_1\cup E_2\cup E_3$ where
 \begin{enumerate}
 
 \item $E_1$ consists of the edges $\{a_i^s,b_j^s\}$ for $s\in \{+,-\}$ 
 and $i=j$ or $j=i+1(\textup{mod }11)$. 
 
 \item $E_2$ consists of the edges $\{a_i^+,b_j^-\}$ for 
 $j=i+3(\textup{mod }11)$ or $j=i+5(\textup{mod }11)$.
 
 \item $E_3$ consists of the edges $\{a_i^-,b_j^+\}$ for $i=j$ or $j=i+2(\textup{mod }11)$.
 
 \end{enumerate}
 \end{defn}

 \begin{lemma}\label{graph}
  $\Gamma$ satisfies the following:
  \begin{enumerate}
   \item The subgraphs spanned by the vertex sets $A_+\cup A_-$ and $B_+\cup B_-$ have no edges, 
 and $A_+\cup B_+$, $A_+\cup B_-$, $A_-\cup B_+$, $A_-\cup B_-$ 
 span subgraphs that are each a cycle.
\item There are no $3$-cycles or $4$-cycles.
  \end{enumerate}
  
 \end{lemma}
 
 \begin{proof}
 $\Gamma$ is a bipartite graph so there are no $3$-cycles. 
Property $(1)$ in the statement of the lemma is easily verified. 
We claim that $\Gamma$ has no $4$-cycles. 
Assume there is a $4$-cycle $C$. 
There are five cases to check. (We write a cycle as a set of edges.)

\begin{enumerate}
 \item $C$ lies in the subgraph spanned by $A_+\cup A_-\cup B_-$. 
 
 So $C$ is of the form $\{\{a_i^+,b_j^-\},\{a_k^-,b_j^-\},\{a_k^-,b_l^-\},\{a_i^+,b_l^-\}\}$.
 
 \item $C$ lies in the subgraph spanned by $A_+\cup A_-\cup B_+$. 
 
 So $C$ is of the form $\{\{a_i^+,b_j^+\},\{a_k^-,b_j^+\},\{a_k^-,b_l^+\},\{a_i^+,b_l^+\}\}$.
 
 \item $C$ lies in the subgraph spanned by $B_+\cup B_-\cup A_-$. 
 
 So $C$  is of the form $\{\{a_j^-,b_i^+\},\{a_j^-,b_k^-\},\{a_l^-,b_k^-\},\{a_l^-,b_i^+\}\}$.
 
 \item $C$ lies in the subgraph spanned by $B_+\cup B_-\cup A_+$. 
 
 So $C$ is of the form $\{\{a_j^+,b_i^+\},\{a_j^+,b_k^-\},\{a_l^+,b_k^-\},\{a_l^+,b_i^+\}\}$.
 
 \item The vertices of $C$ $\{v_1,...,v_4\}$ satisfy $v_1\in A_+,v_2\in B_+,v_3\in A_-,v_4\in B_-$. 
 
 So $C$ is of the form $\{\{a_i^+,b_j^+\},\{a_k^-,b_j^+\},\{a_k^-,b_l^-\},\{a_i^+,b_l^-\}\}$.
 
\end{enumerate}

We treat cases $(1)$ and $(5)$. Cases $(2),(3),(4)$ are similar to $(1)$.

{\bf Case $(1)$}:
Since the distinct vertices $b_l^-,b_j^-$ share a neighbor $a_k^-$, 
we can assume without the loss of generality that $l=k$ and $j=k+1(\textup{mod }11)$. 
So $|i-j|=\pm 1(\textup{mod }11)$.
Now either $l=i+3(\textup{mod }11)$ and $j=i+5(\textup{mod }11)$, 
or $l=i+5(\textup{mod }11)$ and $j=i+3(\textup{mod }11)$. 
In either case we conclude that $|i-j|=\pm 2(\textup{mod }11)$.
This is a contradiction. Therefore such a $4$-cycle cannot exist.
 
 {\bf Case $(5)$}: By construction we observe,
 \begin{enumerate}
  \item $j=i$ or $j=i+1(\textup{mod }11)$.
  
  \item $j=k$ or $j=k+2(\textup{mod }11)$.
  
  \item $l=k$ or $l=k+1(\textup{mod }11)$.
  
  \item $l=i+3(\textup{mod }11)$ or $l=i+5(\textup{mod }11)$.
 \end{enumerate}
 From $(1),(2),(3)$ above we deduce that if $|l-i|\cong n(\textup{mod }11)$, 
 then $n\in \{-2,-1,0,1,2\}$. 
 From $(4)$ on the other hand $|l-i|\cong 3\textup{ or }5(\textup{mod }11)$. 
 This is a contradiction. 
 Therefore we conclude that there are no $4$-cycles in $\Gamma$.
\end{proof}

We remark that in the above construction each set $A_+,A_-,B_+,B_-$ has cardinality $11$,
but for any prime greater than $11$ we can obtain similar constructions.
However, $11$ is the smallest number for which we are able to do such a construction.
 Now we will define our square complex $X$.
 Let $\Theta_1$ be a graph with vertices $v_1,v_2$ 
 and the edge set $E(\Theta_1)=A_-\cup A_+\subseteq V(\Gamma)$, 
 such that each edge meets $v_1$ and $v_2$. 
 The $A_-$ edges are oriented in the direction of $v_2$ 
 (i.e. the arrow points to $v_2$) and the $A_+$ edges 
 are oriented in the direction of $v_1$. 
 Similarly, let  $\Theta_2$ be a graph with vertices $v_3,v_4$,
 and the edge set $E(\Theta_2)=B_-\cup B_+\subseteq V(\Gamma)$. 
 The $B_-$ edges are oriented in the direction 
 of $v_4$ and the $B_+$ edges are oriented in the direction of $v_3$.

Consider the $2$-complex $J=\Theta_1\times \Theta_2$. 
The squares of $J$ are ordered pairs $(a,b)$ 
where $a\in A_+\cup A_-, b\in B_+\cup B_-$. 
We now define a subcomplex $X$ of $J$ in the following manner. 
Let $X^{(1)}=J^{(1)}$. Glue a $2$-face along 
the boundary of every square in $X^{(1)}$ that has 
the property that the corresponding pair $\{a,b\}$ is an edge in $\Gamma$. 
The resulting square complex is $X$.

Identify $S^1$ with $\mathbb{R}/\mathbb{Z}$.
The orientations on the edges of $\Theta_1,\Theta_2$ determine maps $l_i:\Theta_i\to S^1$. 
Under this map each vertex maps to $[0]$, 
and the map on each edge is as follows:
Identify the edge with the unit interval $[0,1]$ 
in a way that the edge is oriented towards the vertex identified with $1$. 
Then define the map $l_i$ on an edge as $x\mapsto [x]$.
Now identify each square $C$ of $X$ isometrically with the unit square $[0,1]^2$, 
where the edges $\{v\}\times [0,1]$ and $[0,1]\times \{v\}$  
are oriented towards $(v,1)$ and $(1,v)$ respectively for each $v\in \{0,1\}$. 
The map $f$ on $C$ is now defined as $(x,y)\mapsto [x+y]$.

The map $f_{\ast}:\pi_1(X)\to \pi_1(S^1)$ is the induced map on the fundamental groups, 
and $\pi_1(X)$ acts on the universal cover $\widetilde{X}$ by deck transformations. 
$f$ lifts to a map $\widetilde{f}:\widetilde{X}\to \mathbb{Z}$, 
which is a $f_{\ast}$-equivariant Morse function on $\widetilde{X}$. 
Conditions $(1),(2),(3)$ of the definition of
a Morse function are apparent, and condition $(4)$ 
follows from the definition of the lift $\widetilde{f}$.

\begin{thm}
 The square complex $X$ has the following properties.
 \begin{enumerate}
\item $\widetilde{X}$ admits a $\textup{CAT}(-1)$ metric.
\item $\textup{Ker}(f_*:\pi_1(X)\to \pi_1(S^1))$ is finitely generated but not finitely presented.
\end{enumerate}
\end{thm}

\begin{proof}
By construction, the link of each vertex is homeomorphic to the graph $\Gamma$. 
Since $\Gamma$ is a flag simplicial complex with no empty-squares, 
$(1)$ follows from Theorem \ref{npc}. 
By construction the ascending and descending links of each vertex are homeomorphic to $S^1$. 
Therefore $(2)$ follows from Theorem \ref{mt}.
\end{proof}

\section{A subgroup of type $F_2$ but not type $F_3$}

We shall construct a three-dimensional cube complex $\Delta$ 
and a piecewise linear function $g:\Delta\to S^{1}$
that lifts to a $g_*$ equivariant Morse function $\widetilde{g}:\widetilde{\Delta}\to \mathbb{R}$
satisfying the hypothesis of Theorem \ref{mt} with $n=2$.
It will be shown that $\widetilde{\Delta}$ is a hyperbolic metric space.
The cube complex $\Delta$ will be constructed as a subcomplex of a product of finite graphs.

We first define graphs $U,V,W$, 
each of which is isomorphic to $K_{22,22}$, 
the complete bipartite graph with $22$ vertices in each ``part''. 
Let the parts of $U,V,W$ be $U_1,U_2$, $V_1,V_2$ and $W_1,W_2$ respectively. 
The vertices of $U_1,V_1,W_1$ are
$\{a_0^+,...,a_{10}^+,a_0^-,...,a_{10}^-\}$
and the vertices of $U_2,V_2,W_2$ are
$\{b_0^+,...,b_{10}^+,b_0^-,...,b_{10}^-\}$.

For each of the graphs $U,V,W$, we fix the following orientations on edges.
Given an edge $\{a_n^s,b_m^t\}$ for $0\leq m,n\leq 10,s,t\in \{+,-\}$,
if $s=t$ then the edge is oriented toward $a_n^s$ otherwise
the edge is oriented toward $b_m^t$.
So any vertex of $U,V,W$ has $11$ incoming and $11$ outgoing edges.
Declare each edge of $U,V,W$ to be isometrically 
identified with the unit interval $[0,1]$ 
in such a way that the edge is oriented towards the vertex identified with $1$.
Let $U\times V\times W$ be the product cube complex. 
We define a cube subcomplex $\Delta$ as follows.

\begin{defn}
 
 Let $(u,v,w)$ be a vertex in $U\times V\times W$. Then $(u,v,w)\in \Delta$ 
 if one of the following holds. 
 (Recall that $\Gamma$ is the graph defined in section $3$.)
 
 \begin{enumerate}
  \item For some $i\in \{1,2\}$, $u\in U_i, v\in V_i,w\in W_i$.
  
  \item $u\in U_1, v\in V_2$ and $\{u,v\}$ is an edge in $\Gamma$.
  
  \item $v\in V_1, w\in W_2$ and $\{v,w\}$ is an edge in $\Gamma$.
  
  \item $u\in U_2, w\in W_1$ and $\{u,w\}$ is an edge in $\Gamma$.
  
 \end{enumerate}
 
 We declare a cell of $U\times V\times W$ to be in $\Delta$ 
 if all its incident vertices are in $\Delta$.
 
\end{defn}

It follows immediately that $\Delta$ is a piecewise Euclidean cube complex. 
Vertices in $U\times V\times W$ that satisfy $(1)$ above are said to be \emph{type 1} vertices, 
  and vertices that satisfy either $(2),(3)$ or $(4)$ are said to be \emph{type 2} vertices.
  It is an easy exercise to show that any two vertices in $\Delta$ are connected
  by a path in $\Delta^{(1)}$. (Check this for type $1$ vertices first.)
  We conclude that $\Delta$ is connected.
 
  The graph $\Omega$ of Figure \ref{Omega} serves as a tool 
  for determining when a given vertex is in $\Delta$.
  The edges of the graph encode the conditions of the definition above as follows:
  Given a vertex $\tau=(u,v,w)$ such that $u\in U_i, v\in V_j, w\in W_k$,
  either $\tau$ is a type $1$ vertex (and hence is in $\Delta$), or 
  there is an edge connecting two of the three nodes $U_i,V_j,W_k$ in the graph $\Omega$.
  Then $\tau\in \Delta$ if and only if the 
  corresponding pair from $u,v,w$ forms an edge in $\Gamma$. 
  
\begin{figure}[!ht]
\centering
\def\svgwidth{0.25\columnwidth}
\def\svgscale{.3}
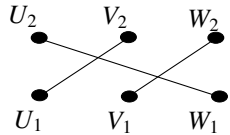
\caption{The graph $\Omega$}
\label{Omega}
\end{figure}

One observes that the map $U\times V\times W\to U\times V\times W;(u,v,w)\mapsto (w,u,v)$
and its iterates are cell permuting isometries
whose restriction to $\Delta$ induces a map $\Delta\to \Delta$ that is a cell permuting isometry.
These symmetries of our construction will be invoked in arguments that follow.
  
 The orientations on the edges 
 induce a PL function $f:\Delta\to S^1$, 
 which is $x\to [x]$ on each edge, and on the product it is defined as $(x,y,z)\to [x+y+z]$. 
 (Recall the identification of
  each edge with $[0,1]$ that was described above, 
  and the identification of $S^1$ with $\mathbb{R}/\mathbb{Z}$.) 
  The map $f$ lifts as a PL Morse function between 
  the universal covers $\widetilde{f}:\widetilde{\Delta}\to \mathbb{R}$. 
  Now we prove two key lemmas, the first of which examines the links of vertices.
  
  \begin{lemma}\label{main}
  Let $\tau\in \Delta$ be a vertex. Then the following holds.
  (Here $\star$ denotes the topological join and $\Upsilon$ is a discrete set of four points.)
  \begin{enumerate}
   \item If $\tau$ is a type $1$ vertex then $Lk(\tau,\Delta)$ 
   is homeomorphic to $\Upsilon\star \Upsilon\star \Upsilon$.
   
   \item  If $\tau$ is a type $2$ vertex then 
   $Lk(\tau,\Delta)$ is homeomorphic to $\Gamma\star \Upsilon$.
  \end{enumerate}
  In particular, $\Delta$ is nonpositively curved.
  Furthermore, in either case 
  $Lk^{\uparrow}(\tau,\widetilde{\Delta}),Lk^{\downarrow}(\tau,\widetilde{\Delta})$ 
  are both homeomorphic to $S^2$.

  \end{lemma}

 \begin{proof}
  
  Let $\tau=(u,v,w)$ be a type $1$ vertex.
  Assume that $u\in U_1,v\in V_1,w\in W_1$.
  Recall that $u,v,w\in \{a_0^+,...,a_{10}^+,a_0^-,...,a_{10}^-\}\subseteq V(\Gamma)$.
  Let the neighbors of $u,v,w$ in $\Gamma$ be 
  $$\{b_{k_1}^+,b_{k_2}^+,b_{k_3}^-,b_{k_4}^-\},
  \{b_{l_1}^+,b_{l_2}^+,b_{l_3}^-,b_{l_4}^-\},\{b_{j_1}^+,b_{j_2}^+,b_{j_3}^-,b_{j_4}^-\}$$
  respectively.
  The following are the $1$-cells adjacent to $\tau$ in $\Delta$:
  \begin{enumerate}
   \item $[u,b_{j_i}^+]\times v\times w$ for $1\leq i\leq 2$ and 
   $[u,b_{j_i}^-]\times v\times w$ for $3\leq i\leq 4$.
   \item $u\times [v,b_{k_i}^+]\times w$ for $1\leq i\leq 2$ 
   and $u\times [v,b_{k_i}^-]\times w$ for $3\leq i\leq 4$.
   \item $u\times v\times [w,b_{l_i}^+]$ for $1\leq i\leq 2$
   and $u\times v\times [w,b_{l_i}^-]$ for $3\leq i\leq 4$.
  \end{enumerate}
  
  It follows that $Lk(\tau,\Delta\cap (U\times v\times w))$,
  $Lk(\tau,\Delta\cap (u\times V\times w))$, $Lk(\tau,\Delta\cap (u\times v\times W))$
  are all discrete sets of four points each.
  Furthermore, it follows from the definition of $\Delta$ that
  $[u,p]\times [v,q]\times [w,r]$ is a cube in $\Delta$ for 
  each $p\in \{b_{j_1}^+,b_{j_2}^+,b_{j_3}^-,b_{j_4}^-\},
  q\in \{b_{k_1}^+,b_{k_2}^+,b_{k_3}^-,b_{k_4}^-\},
  r\in \{b_{l_1}^+,b_{l_2}^+,b_{l_3}^-,b_{l_4}^-\}$.
  So $Lk(\tau,\Delta)$ is the topological join of these sets.
  Observe that exactly two of the four $1$-cells in each of $(1),(2),(3)$
  are oriented away from $\tau$ and the remaining are oriented towards $\tau$.
  So it follows immediately that 
  $Lk^{\uparrow}(\tau,\widetilde{\Delta}),Lk^{\downarrow}(\tau,\widetilde{\Delta})$
  are homeomorphic to $S^2$.
  The case where $u\in U_2,v\in V_2,w\in W_2$ is similar.

  Now consider the case where $\tau=(u,v,w)\in \Delta$ is a type $2$ vertex.
  Assume that $u\in U_2,v\in V_1,w\in W_1$.
  The $1$-cells incident to $\tau$ in $\Delta$ are:
  \begin{enumerate}
   \item $[u,a_i^s]\times v\times w$ for $0\leq i\leq 10$, $s\in \{+,-\}$.
   \item $u\times [v,b_i^s]\times w$ for $0\leq i\leq 10$, $s\in \{+,-\}$.
   \item $u\times v\times [w,p]$,
   for $p\in \{b_{n_1}^+,b_{n_2}^+,b_{n_3}^-,b_{n_4}^-\}$
   where $\{b_{n_1}^+,b_{n_2}^+,b_{n_3}^-,b_{n_4}^-\}$
   are the four neighbors of $v$ in $\Gamma$.
   \end{enumerate}
   
  Now given $1$-cells $[u,a_i^s]\times v\times w$, $u\times [v,b_j^t]\times w$,
  observe that there is a square $[u,a_i^s]\times [v,b_j^t]\times w$
  in $\Delta$ if and only if $a_i^s,b_j^t$ are connected in $\Gamma$ by an edge.
  This means that 
  $Lk(\tau,\Delta\cap (U\times V\times w))\cong \Gamma$.
  Furthermore, the definition of $\Delta$ implies that 
  $[u,a_i^s]\times [v,b_j^t]\times [w,b_k^r]$
  is a cube in $\Delta$ if and only if $[u,a_i^s]\times [v,b_j^t]\times w$
  is a square in $\Delta$ and $u\times v\times [w,b_k^r]$ is a $1$-cell in $\Delta$.
  This means that $Lk(\tau,\Delta)$ is the topological join of $\Gamma$ and the 
  discrete set of four points, $Lk(\tau,\Delta\cap (u\times v\times W))$.
  
  Now $Lk^{\uparrow}(\tau,\Delta\cap (U\times V\times w))$,
  $Lk^{\downarrow}(\tau,\Delta\cap (U\times V\times w))$
  are both cycles,
  and $Lk^{\uparrow}(\tau,\Delta\cap (u\times v\times W))$,
  $Lk^{\downarrow}(\tau,\Delta\cap (u\times v\times W))$
  are both discrete sets of two points each.
  This means that $Lk^{\uparrow}(\tau,\widetilde{\Delta})\cong S^2$
  and $Lk^{\downarrow}(\tau,\widetilde{\Delta})\cong S^2$.
  For any other vertex $\tau$ of type $2$,
  the analysis is similar by symmetry of the construction.
\end{proof}
  So far we have shown the following.
  \begin{enumerate}
   \item $\Delta$ is a nonpositively curved cube complex.
   \item By Lemma \ref{main}
    and Theorem \ref{mt} it follows that
   $Ker(f_*:\pi_1(\Delta)\to \pi_1(S^1))$ is finitely presented but not of type $F_3$.
  \end{enumerate}

  Now we will show that $\widetilde{\Delta}$ is a hyperbolic metric
  space and hence $\pi_1(\Delta)$ is a hyperbolic group. 
  We have already established that $\widetilde{\Delta}$ is a $CAT(0)$
  space, and so by Theorem \ref{flatplanes} it suffices to show
  that $\widetilde{\Delta}$ does not contain isometrically embedded flat planes.
  
  Let $\tau=(u,v,w)$ be a type $2$ vertex in $\Delta$.
   From the previous lemma $Lk(\tau,\Delta)\cong \Gamma \star \Upsilon$,
   where $\Upsilon$ is a discrete set of four points.
   So $Lk(\tau,\Delta^{(1)})$ is naturally identified with $V(\Gamma)\cup \Upsilon$.
  
  \begin{lemma}\label{sedge}
   Let $\tau,\tau^{\prime}$ be type $2$ vertices in $\Delta$
   such that $[\tau,\tau^{\prime}]$ is
   a $1$-cell in $\Delta$.
   The $Lk(\tau,[\tau,\tau^{\prime}])\in \Upsilon$ if and only if
   $Lk(\tau^{\prime},[\tau,\tau^{\prime}])\in \Upsilon$.
  \end{lemma}
  
  \begin{proof}
  We assume that $\tau=(u,v,w),\tau^{\prime}=(u^{\prime},v,w)$.
   Let $u\in U_i,u^{\prime}\in U_j$ where $\{i,j\}=\{1,2\}$.
   Also, let $v\in V_k, w\in W_l$.
   Assume that $Lk(\tau,[\tau,\tau^{\prime}])\in \Upsilon$.
   It follows that $Lk(\tau,\Delta\cap (U\times v\times w))=\Upsilon$.
   So $U_j$ is connected by an edge in $\Omega$ with either $V_k$ or $W_l$.
   We will show that $U_i$ is connected by an edge with either $V_k$
   or $W_l$ in $\Omega$, and hence $Lk(\tau^{\prime},[\tau,\tau^{\prime}])\in \Upsilon$.
   
   Assume that this is not the case.
   Then since $\tau^{\prime}$
   is a type $2$ vertex it must be the case that $V_k,W_l$ 
   are connected by an edge in $\Omega$. 
   This cannot be true since $U_j$ is connected by an edge with
   either $V_k$ or $W_l$.
   This proves our assertion.
   By symmetry of our construction this
   follows for any arbitrary type $2$ vertex in $\Delta$.
   \end{proof}
  
  \begin{defn}
   A $1$-cell $[\tau,\tau^{\prime}]$ in $\Delta$ satisfying the statement of Lemma \ref{sedge} 
   i.e.,
   $Lk(\tau,[\tau,\tau^{\prime}])\in \Upsilon$ and 
   $Lk(\tau^{\prime},[\tau,\tau^{\prime}])\in \Upsilon$
   is called a \emph{special $1$-cell}.
   Denote the union of all special $1$-cells in $\Delta$ by $L$.
   A lift of a special $1$-cell in $\widetilde{\Delta}$
   is a special $1$-cell in $\widetilde{\Delta}$ and $\widetilde{L}$
   is the union of all special $1$-cells in $\widetilde{\Delta}$.
  \end{defn}

  Figure \ref{cube} depicts a cube in $\Delta$. 
    The three bold $1$-cells are the special $1$-cells,
    the vertices $\tau_1,\tau_2$ are the type $1$ vertices
    and the remaining vertices are of type $2$.

\begin{figure}[!ht]
\centering
\def\svgwidth{0.3\columnwidth}
\def\svgscale{.3}
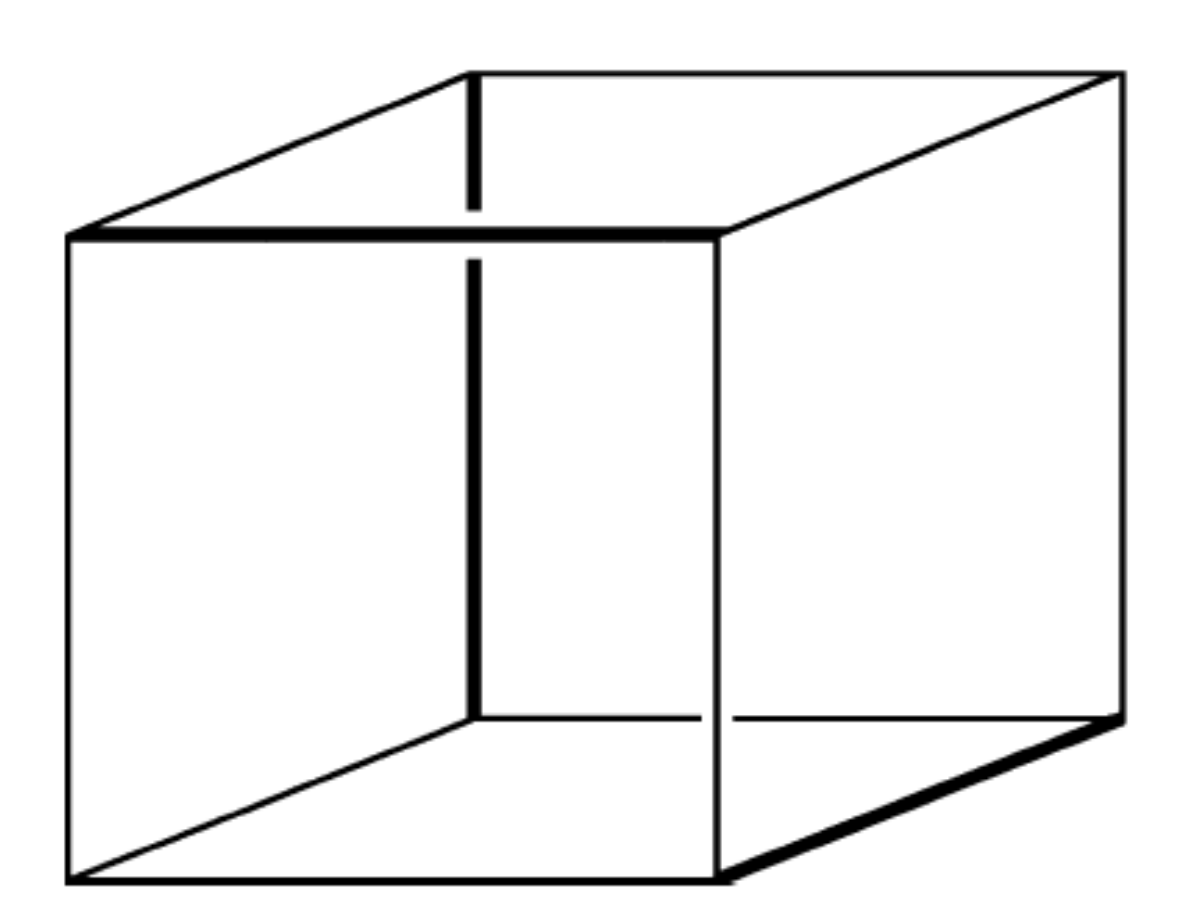
\caption{}
\label{cube}
\end{figure}
\begin{lemma}
   $\widetilde{\Delta}$ does not contain isometrically embedded flat planes, 
   and hence is a hyperbolic metric space.
\end{lemma}

  \begin{proof}
   The proof is similar to the proof of $6.1(3)$ in \cite{Brady},
   and we claim no originality here. 
   We adapt that argument to our construction.
   Let us assume that there is an isometric embedding $i:\mathbb{R}^2\to \widetilde{\Delta}$. 
    We say that a point $x\in i(\mathbb{R}^2)$ is a \emph{transverse intersection point} 
    if there is a neighborhood $U$ of $x$ in $i(\mathbb{R}^2)$ such that the 
    intersection of $U$ with $\widetilde{L}$ is the point $x$. 
    Following \cite{Brady}, our proof is divided into two steps.
    In \emph{step $1$} we will show that $i(\mathbb{R}^2)$ has a transverse intersection point.
    In \emph{step $2$} we will show that the angle around the transverse intersection
    point in $i(\mathbb{R}^2)$ is greater than $2\pi$, contradicting the fact that
    this is an isometric embedding.

{\bf Step 1}:
Let $C$ be a cube in $\widetilde{\Delta}$ such that $i(\mathbb{R}^2)\cap C$ is nonempty
and two dimensional. (Such cubes must exist in $\widetilde{\Delta}$ since
$i(\mathbb{R}^2)$ is an isometric embedding.)There are four cases to consider.

In the first case, $i(\mathbb{R}^2)$
intersects a special $1$-cell $e$ of $C$ in a vertex $p$.
Now $i(\mathbb{R}^2)$ must also intersect a neighboring cube $C^{\prime}$ of $C$ 
that shares a $2$-face with $C$ and
contains a special $1$-cell $e^{\prime}$ incident to $p$.
Then either $i(\mathbb{R}^2)$ contains $e^{\prime}$ or $p$ is a transverse intersection point.
Since $i(\mathbb{R}^2)$ is an isometric embedding, it cannot contain $e^{\prime}$
or else it would also contain $e$.
In the second case, $i(\mathbb{R}^2)$
intersects a special $1$-cell of $C$ in
an interior point,
in which case it is clear that this is a transverse intersection point.
In the third case,
$i(\mathbb{R}^2)$ contains a special $1$-cell of $C$.
In this case it transversely intersects a different special $1$-cell of $C$.
(Recall that $i(\mathbb{R}^2)\cap C$ is two dimensional and see Figure \ref{cube}).

Finally, consider the case where $i(\mathbb{R}^2)$ does not intersect a special $1$-cell
of $C$.
Then $i(\mathbb{R}^2)$ intersects a $1$-cell incident to a type $1$ vertex $\tau$
in $C$.
Let $J$ be the subcomplex of $\widetilde{\Delta}$ consisting of cubes in $\widetilde{\Delta}$ that have a 
nonempty intersection with $i(\mathbb{R}^2)$.
Recall that $Lk(\tau,\widetilde{\Delta})\cong \Upsilon\star \Upsilon\star \Upsilon$
where $\Upsilon$ is a discrete set of four points.
The set of cubes incident
to $\tau$ in $J$ is a subcomplex of a stack of $8$ cubes depicted in
Figure \ref{pattern}.
($C$ is one of the $8$ cubes.)
\begin{figure}[!ht]
\centering
\def\svgwidth{0.3\columnwidth}
\def\svgscale{.3}
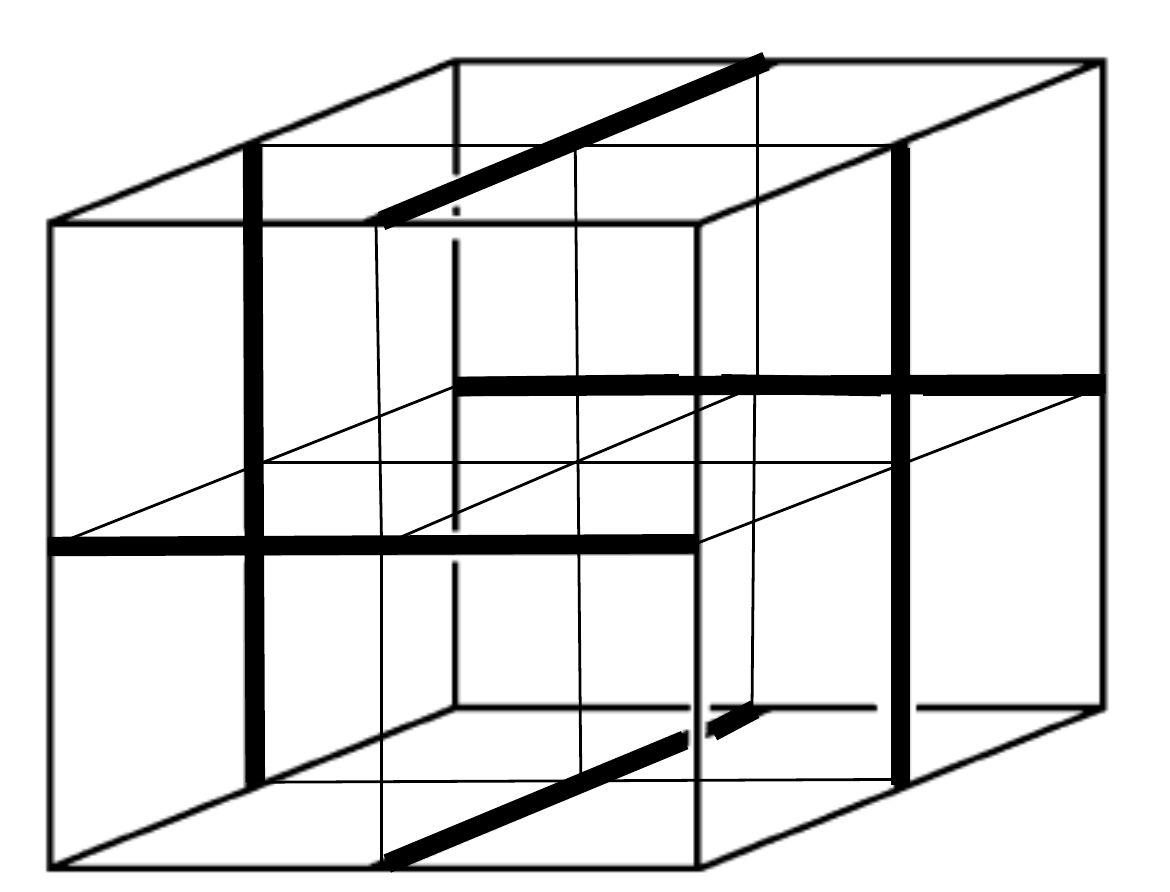
\caption{}
\label{pattern}
\end{figure}
Here the bold $1$-cells are the special $1$-cells.
Now by figure \ref{pattern} it must be the case that $i(\mathbb{R}^2)$ transversely 
intersects a special $1$-cell in a neighboring cube of $C$.

{\bf Step 2}: Now we demonstrate a contradiction to our
assumption that $i(\mathbb{R}^2)$ is an isometrically embedded flat plane in $\widetilde{\Delta}$.
This involves computing the angle in $i(\mathbb{R}^2)$ around a transverse intersection point 
$p$.

Let $Y$ be the subcomplex consisting of cubes in $\widetilde{\Delta}$ that contain $p$
and for which $C\cap i(\mathbb{R}^2)$ is $2$-dimensional.
For each cube $C$ in $Y$, $i^{-1}(i(\mathbb{R}^2)\cap C)$ is a polygon,
and $i^{-1}(i(\mathbb{R}^2)\cap Y)$ is a union of polygons in $\mathbb{R}^2$ that are incident
to $i^{-1}(p)$ such that the sum of the angles at $i^{-1}(p)$ in each polygon
is $2\pi$.

Let $C,C^{\prime}$ be cubes in $Y$
such that $i^{-1}(i(\mathbb{R}^2)\cap C),i^{-1}(i(\mathbb{R}^2)\cap C^{\prime})$ 
are adjacent polygons.
Then $i(\mathbb{R}^2)\cap (C\cup C^{\prime})$ is locally the intersection
of $i(\mathbb{R}^2)$ with a Euclidean half space.
This means that the angle sum of any two consecutive polygons in
$i^{-1}(i(\mathbb{R}^2)\cap Y)$ is $\pi$.
To establish a contradiction, it suffices to show that
the number of polygons in $i^{-1}(i(\mathbb{R}^2)\cap Y)$ is greater than $4$.
Let $e=[\tau,\tau^{\prime}]$ be a special $1$-cell containing $p$.
Recall that since $\tau$ is a type $2$ vertex, $Lk(\tau,\widetilde{\Delta})$
is naturally identified with $\Gamma\star \Upsilon$.
Now $Lk(\tau,Y)$ is a subcomplex of $\Gamma\star \Upsilon$
and by Lemma \ref{sedge} we know that $Lk(\tau,[\tau,\tau^{\prime}])\in \Upsilon$.
So there is a natural bijection between the aforementioned set of polygons
and the set of edges of a cycle in 
$\Gamma$.
Since all cycles in $\Gamma$
have more than four edges,
we have established that the angle around $p$ in $i(\mathbb{R}^2)$ is greater than $2\pi$
contradicting the fact that this is an isometric embedding.
\end{proof}

\section{Concluding remarks}

At this point we do not have a concrete way of distinguishing the groups
in Section $4$ from Brady's groups in \cite{Brady}, other than the method of construction.
(Our complexes are different from Brady's complexes since the links are different.)
In fact, there is a striking similarity between the two examples, even though the methods
of construction are entirely different.
Nevertheless, we do believe that our approach is less abstract.

This construction does not seem to have a natural generalization in higher dimensions.
It seems likely that any natural generalization in dimensions four and higher always
produces flat planes in the universal cover.
As a result, it is not clear whether such an approach can be used to 
construct hyperbolic groups with subgroups that are of type $F_n$ but not of type $F_{n+1}$
for $n>2$.

\end{document}